\documentclass{article}

\usepackage[utf8]{inputenc}
\usepackage{amsmath,amssymb,amsthm}
\usepackage[backend=bibtex,firstinits=true]{biblatex}
\usepackage{hyperref}

\DeclareMathOperator{\ev}{ev}
\DeclareMathOperator{\sign}{sign}
\DeclareMathOperator{\Aut}{Aut}
\DeclareMathOperator{\codim}{codim}

\newcommand{\Mbar}{\overline M}
\newcommand{\CC}{\mathbb C}
\newcommand{\QQ}{\mathbb Q}
\newcommand{\ZZ}{\mathbb Z}
\newcommand{\PP}{\mathbb P}
\newcommand{\vir}{\mathrm{vir}}

\newtheorem{theorem}{Theorem}
\newtheorem{proposition}{Proposition}
\newtheorem{lemma}{Lemma}
\newtheorem{question}{Question}
\newtheorem*{corollary}{Corollary}
\newtheorem*{fact}{Fact}

\title{Gromov--Witten theory of target curves and the tautological
  ring}

\author{Felix Janda \thanks{Partially supported by the Swiss National
    Science Foundation grant SNF 200021\_143274}}

\begin{document}

\maketitle
\begin{abstract}
  In the Gromov--Witten theory of a target curve, we consider descendent
  integrals against the virtual fundamental class relative to the
  forgetful morphism to the moduli space of curves.
  We show that cohomology classes obtained in this way lie in the
  tautological ring.
\end{abstract}

\setcounter{section}{-1}
\section{Introduction}

Let $X$ be a nonsigular projective variety over $\CC$, let
$\Mbar_{g, n}(X, \beta)$ be the moduli space of stable maps to $X$ of
class $\beta$ and let
\begin{equation*}
  \pi\colon \Mbar_{g, n}(X, \beta) \to \Mbar_{g, n}
\end{equation*}
be the forgetful map to the moduli space of stable
curves.\footnote{See \cite{FuPa97} for an introduction to the moduli
  space of stable maps.}
The moduli space $\Mbar_{g, n}(X, \beta)$ possess a perfect
obstruction theory defining a virtual fundamental class
\begin{equation*}
  [\Mbar_{g, n}(X, \beta)]^\vir \in H_*(\Mbar_{g, n}(X, \beta), \QQ)
\end{equation*}
of expected dimension (\cite{BeFa97, LiTi98}).

The tautological rings $RH^*(\Mbar_{g, n})$ of $\Mbar_{g, n}$ are most
compactly defined (see \cite{FaPa05}) as the smallest system of
subrings of $H^*(\Mbar_{g, n})$\footnote{We will generally work with
  $\QQ$-valued homology and cohomology.}
stable under push-forward and pullback by the maps
\begin{itemize}
\item $\Mbar_{g, n + 1} \to \Mbar_{g, n}$ forgetting one of the
  markings,
\item
  $\Mbar_{g_1, n_1 + 1} \times \Mbar_{g_2, n_2 + 1} \to \Mbar_{g_1 +
    g_2, n_1 + n_2}$ gluing two curves at a point,
\item $\Mbar_{g - 1, n + 2} \to \Mbar_{g, n}$ gluing together two
  markings of a curve.
\end{itemize}
While this definition seems restrictive many geometric classes lie in
the tautological ring.

\begin{question}[Pandharipande \cite{FaPa05}]
  Does
  \begin{equation*}
    \pi_* [\Mbar_{g, n}(X, \beta)]^\vir \in  H_*(\Mbar_{g, n}) \cong H^*(\Mbar_{g, n})
  \end{equation*}
  lie inside $RH^*(\Mbar_{g, n})$ when $X$ is defined over the field
  $\overline\QQ$ of algebraic numbers?
\end{question}
This question can be answered affirmative for toric varieties by the
method of virtual localization \cite{GrPa99}.
We will show in this article that the answer is also ``yes'' in the
case when $X$ is a curve.
For this, it is convenient to generalize the question.

In the first direction of generalization, we also allow arbitrary
descendent insertions: For each marking there is an evaluation map
$\ev_i\colon \Mbar_{g, n}(X, \beta) \to X$.
Furthermore, for each marking let $\psi_i$ be the first Chern class of
the cotangent line bundle at marking $i$.
For any choice of $n$ cohomology classes
$\gamma_1, \dotsc, \gamma_n \in H^*(X)$ and non-negative integers
$k_1, \dotsc, k_n$, more general cohomology classes can be defined by
\begin{equation*}
  \pi_*\left(\prod_{i = 1}^n \psi_i^{k_i} \ev_i^*(\gamma_i) \cap
    [\Mbar_{g, n}(X, \beta)]^\vir\right) \in H^*(\Mbar_{g, n}).
\end{equation*}
We will call such a class a ``GW-class''.
Integrating a GW-class gives the corresponding usual Gromov--Witten
descendent invariant.
We can again ask whether GW-classes lie inside $RH^*(\Mbar_{g, n})$.

Another direction of generalization is possible via relative
Gromov--Witten theory \cite{LiRu01, Li02, IoPa03}, where for a smooth
variety $X$ together with a smooth divisor $D$, the moduli space
$\Mbar_{g, n}(X, \beta, \mu_1, \dotsc, \mu_m)$ of relative stable maps
is considered.
This moduli space is a compactification of the space of stable maps to
$X$ such that the preimage of $D$ is finite and the cohomology valued
partitions $\mu_1, \dotsc, \mu_m$ specify for each connected component
$D'$ of $D$ the ramification profile over $D'$ and the class of the
source curve in $D'$.
We will follow the convention that the preimages of $D$ are marked, so
that a projection map
\begin{equation*}
  \pi\colon \Mbar_{g, n}(X, \beta, \eta_1, \dotsc, \eta_m) \to \Mbar_{g, n + \ell(\eta)}
\end{equation*}
can be defined where $\ell(\eta)$ is the sum of the lengths of the
partitions $\eta_1, \dotsc, \eta_m$.
\begin{question}
  Does
  \begin{equation*}
    \pi_* \left(\prod_{i = 1}^n \psi_i^{k_i} \ev_i^*(\gamma_i) \cap [\Mbar_{g, n}(X, \beta, \mu_1, \dotsc, \mu_m)]^\vir\right) \in  H^*(\Mbar_{g, n + \ell(\eta)})
  \end{equation*}
  lie inside $RH^*(\Mbar_{g, n})$ when $X$ and $D$ are defined over
  $\overline\QQ$?
\end{question}

\smallskip

The main result of this paper is the following theorem.
\begin{theorem}
  \label{thm:thm}
  If $X$ is an algebraic curve and $D$ a collection of pairwise
  distinct points on $X$, all (relative) GW-classes lie in the
  tautological ring $RH_*(\Mbar_{g, n})$.
\end{theorem}
To say more about this result, we specialize the discussion to the case
when $X$ is a curve.

Recall that the cohomology of an algebraic curve $X$ of genus $h$ over
$\CC$ has a basis
\begin{equation*}
  \{1, \alpha_1, \dotsc, \alpha_h, \beta_1, \dotsc, \beta_h, \omega\}
\end{equation*}
such that $1$ is the identity of the cup product, $\omega$ is the
Poincaré dual of a point and the $\alpha_i \in H^{1, 0}(X, \CC)$ and
$\beta_i \in H^{0, 1}(X, \CC)$ form a symplectic basis of
$H^1(X, \CC)$, i.e.\ $\alpha_i \cup \beta_i = \omega$,
$\beta_i \cup \alpha_i = -\omega$ for all $i$, and all other cup
products vanish.

The special case of Theorem~\ref{thm:thm} when $X = \PP^1$ was proven
in \cite{FaPa05}.
A large part of our proof will be a reduction to this special case.
We also build on the series of articles \cite{OkPa06a, OkPa06b,
  OkPa06c}, which gives an effective way to calculate all relative
Gromov--Witten invariant of $X$.

\smallskip

We now give an overview of the proof of Theorem~\ref{thm:thm}.
If the GW-class has only even cohomology insertions, we can use the
degeneration formula to calculate the GW-class in terms of GW-classes
of $\PP^1$ relative to a point and by the results of \cite{FaPa05}
these are also tautological.
This will be done in Section~\ref{sec:even}.

In the presence of odd insertions new phenomena can occur.
For example we might obtain odd classes in $H^*(\Mbar_{g, n})$.
These can only be tautological if they vanish, since by definition
tautological classes are algebraic.
More generally, one might obtain classes inside a piece
$H^{i, j}(\Mbar_{g, n})$ of the Hodge diamond with $i \neq j$.
We call such classes \emph{unbalanced}.
\begin{corollary}
  All unbalanced GW-classes of curves vanish.
\end{corollary}
Actually we will first prove this corollary in
Section~\ref{sec:unbalanced} and use it as an input for the proof of
Theorem~\ref{thm:thm}.

The remaining GW-classes are \emph{balanced}.
In Section~\ref{sec:balanced} we will give an algorithm to calculate
such GW-classes in the presence of odd cohomology in terms of
GW-classes with only even insertions.
It is a straightforward generalization of the algorithm given in
\cite{OkPa06c}.

\smallskip

If there are odd insertions, we cannot use a degeneration formula to
reduce to the case of $\PP^1$.
Still, it is possible to deform $X$ into a chain of elliptic curves to
reduce to the genus 0 and genus 1 cases.
This is done in Section~\ref{sec:red}.
Therefore, starting from Section~\ref{sec:rel}, we will assume $X$ to
be of genus one.

As in \cite{OkPa06c}, we will use the following properties of
Gromov--Witten theory to relate GW-classes with odd insertions to those
with only even insertions.
\begin{itemize}
\item algebracity of the virtual fundamental classes
\item invariance under monodromy transformations of $X$
\item degeneration formulae
\item vanishing relations from the group structure on an elliptic curve
\end{itemize}

We will study relations coming from the monodromy invariance of
Gromov--Witten theory and the group structure of an elliptic curve in
Section~\ref{sec:mono} and \ref{sec:ell}, respectively.

For the proof of the Corollary, we will only need the results from
Section~\ref{sec:red}, \ref{sec:mono} and \ref{sec:unbalanced}.
It is even possible to adapt the proof so that the use the reduction
to genus 1 is not necessary.
Its proof is the main new part of this article.

We have tried to apply the analogous methods in the case when $X$ is a
quintic surface but they do not seem to suffice in this case.

\subsection*{Notations and Conventions}

Because of our extensive use of the degeneration formula, we will
always allow our source curves to be disconnected.
So we will always work with disconnected Gromov--Witten invariants,
GW-classes and a tautological ring of not necessarily connected
curves.
These can however be related to their connected counterparts in a
purely combinatorial fashion.

The only possible curve classes of $X$ are multiples of the class of
$X$.
We will write $\Mbar_{g, n}(X, d)$ for the space of stable maps of
curve class $d[X]$ or, in other words, of degree $d$.

We will use the notation
\begin{equation*}
  \left\langle \tau_{k_1}(\gamma_1) \dots \tau_{k_n}(\gamma_n) | \eta\right\rangle^X_{g, d} := \int_{[\Mbar_{g, n}(X, \eta)]^\vir} \prod_{i = 1}^n \psi_i^{k_i} \ev_i^*(\gamma_i)
\end{equation*}
for relative Gromov--Witten invariants where
$\eta = (\eta_1, \dotsc, \eta_m)$ is a collection of partitions.
For GW-classes we use the analogous but non-standard notation
\begin{multline*}
  [ \tau_{k_1}(\gamma_1) \dots \tau_{k_n}(\gamma_n) | \eta]^X_r := \\
  \pi_*\left(\prod_{i = 1}^n \psi_i^{k_i} \ev_i^*(\gamma_i) \cap
    [\Mbar_{g, n}(X, \eta)]^\vir\right) \in
  H_{2r}(\Mbar_{g, n + \ell(\eta)}).
\end{multline*}
We have left out the degree $d$ since it is the size of any of the
usual partitions $\eta_1, \dotsc, \eta_m$.
The notation also does not specify the genus $g$ of the source curve
since it is determined by the formula
\begin{equation*}
  r = 2g - 2 + n + d(2 - 2h) - \sum_{i = 1}^n (k_i + \codim(\gamma_i)) - \sum_{i = 1}^m (d - \ell(\eta_i)).
\end{equation*}
If this would lead to a half-integer value of $g$, we define the
GW-class to be zero.

\subsection*{Acknowledgments}

This work was carried out while being a PhD student at ETH Zürich
under the supervision of R.~Pandharipande.
I would like to thank him for the introduction to the problem, his
support and many helpful discussions.

\section{Even classes}
\label{sec:even}

We consider the computation of GW-classes of a curve $X$ with only even
insertions.

There is a nonsingular family $X_t$ of curves of genus $h > 0$ over
$\CC$ such that $X_t \cong X$ for $t \neq 0$ and $X_0$ is an
irreducible curve of geometric genus $h - 1$ with a node.
The degeneration formula relates the GW-classes of $X$ to the
GW-classes of the normalization $\tilde X_0$ of $X_0$ relative to the
two preimages of the marked point.
It is important to note that the even classes of $X$ can be lifted to
$\tilde X_0$.
All of this discussion generalizes to the situation of $X$ relative to
marked points $q_1, \dotsc, q_m$.

Let
\begin{equation*}
  M = \prod_{h \in H} \tau_{o_h}(1) \prod_{h' \in H'} \tau_{o'_{h'}}(\omega)
\end{equation*}
be a monomial in insertions of even classes and let
$\eta_1, \dotsc, \eta_m$ be choices of splittings at the relative
points.
Since the target curve is irreducible the degeneration formula
\cite{Li02} in this case says that
\begin{equation*}
  \left[ M | \eta_1, \dotsc, \eta_m \right]_r^X = \sum_{|\mu| = d} \mathfrak z(\mu)\ \iota_* \left[ M | \eta_1, \dotsc, \eta_m, \mu, \mu\right]_r^{\tilde X_0},
\end{equation*}
where the sum is over partitions
$\mu = (\mu_1, \dotsc, \mu_{\ell(\mu)})$, and the automorphism factor
$\mathfrak z(\mu)$ is defined by
\begin{equation*}
  \mathfrak z(\mu) = |\Aut(\mu)| \prod_{i = 1}^{\ell(\mu)} \mu_i,
\end{equation*}
and $\iota$ is the map gluing together the last two markings.

By using this formula repeatedly, we can reduce the genus $h$ until we
arrive at the case of $X = \PP^1$ relative to $q_1, \dotsc, q_n$,
which has been studied in \cite{FaPa05}.
This implies that Theorem~\ref{thm:thm} is true in the case that all
$\gamma_i$ are even classes.

\section{Reduction to genus one}
\label{sec:red}

Recall that we have chosen a symplectic basis
$\alpha_i, \beta_i \in H^1(X, \CC)$.
There is a deformation $Y \to \PP^1$ of $X$ into
$\tilde X = E \cup X'$, a curve of genus one and a curve of genus
$h - 1$ connected at a node $p$.
Moreover, the symplectic basis of $H^1(X, \CC)$ can be lifted to $Y$
such that over $\tilde X$ the classes $\alpha_1, \beta_1$ give a
symplectic basis of $H^1(E, \CC)$ and the other $\alpha_i$ and
$\beta_i$ give a symplectic basis of $H^1(X', \CC)$.
Furthermore, the deformation can be chosen such that $\omega$ deforms
to the Poincaré dual class of a point on the genus 1 curve.
Similarly, in the relative theory, the deformations of the relative
points $q_1, \dotsc, q_m$ can be assumed to lie on the genus 1
component.

The degeneration formula is slightly more complicated to write down in
this case since there is a choice for the splitting of the domain
curve into two parts, one for each component of $\tilde X$, and a
choice of splitting $\mu$ at $p$.
For each partition $g = g_1 + g_2 + \ell(\mu) - 1$ of $g$, there is a
gluing map
\begin{equation*}
  \iota\colon \Mbar_{g_1, n_1 + \ell(\eta) + \ell(\mu)} \times \Mbar_{g_2, n_2 + \ell(\mu)} \to \Mbar_{g, n_1 + n_2 + \ell(\eta)},
\end{equation*}
gluing two curves along the last $\ell(\mu)$ markings

Let $M_\omega$, $M_1$, $M_2$ be monomials in insertions of elements in
$\{\omega\}$, $\{\alpha_1, \beta_1\}$ and
$\{\alpha_i, \beta_i | i \neq 1\}$ respectively.
Furthermore, let $M := \tau_{o_H}(1)$ where
\begin{equation*}
  \tau_{o_H}(1) := \prod_{h \in H} \tau_{o_h}(1)
\end{equation*}
be a monomial in insertions of the identity.
After a change of sign, a general GW-class we wish to calculate is of
the form
\begin{equation*}
  \left[ MM_\omega M_1M_2 | \eta_1, \dotsc, \eta_m \right]^X_r.
\end{equation*}
By the degeneration formula this equals
\begin{equation*}
  \sum_{\substack{r_1 + r_2 = r, \\ |\mu| = d, I \subset H}} \mathfrak z(\mu)
  \iota_*\left([ \tau_{o_I}(1) M_\omega M_1 | \eta_1, \dotsc, \eta_m, \mu ]^E_{r_1}, [ \tau_{o_{H \setminus I}}(1) M_2 | \mu ]^{X'}_{r_2}\right).
\end{equation*}

Since the tautological rings are compatible with $\iota$, we can induct
on the genus of $X$ to reduce to the case where $X$ is of genus 1.
Let us fix a symplectic basis $\alpha, \beta$ of $H^1(X, \CC)$ for
this and the following sections.
In this case, we can use a different degeneration to simplify the
problem further.
Namely, $X$ can be degenerated to $X$ with a rational tail.
This can be used to move the $\omega$ insertions and all but one
relative point to the rational tail.

We have therefore reduced the proof of Theorem~\ref{thm:thm} to
showing the following statements.

\begin{theorem}
  \label{thm:red}
  Let $X$ be a curve of genus 1 relative to a point $p$ with
  symplectic basis $\alpha, \beta \in H^1(X, \CC)$.
  Then for every partition $\eta$ of $d$ and any monomial $M$ in
  insertions of identity classes, $\alpha$ and $\beta$ the classes
  \begin{equation*}
    \left[ M| \eta \right]_r^X
  \end{equation*}
  lie in the tautological ring $RH^*(\Mbar_{g, n + \ell(\eta)})$.
  In particular, if the number of insertions of $\alpha$ does not
  equal the number of insertions of $\beta$, the class is zero.
\end{theorem}

\section{Relations}
\label{sec:rel}

In this section we introduce two suitably generalized methods of
\cite{OkPa06c} to produce relations between relative GW-classes of
genus one targets.

\subsection{from monodromy}
\label{sec:mono}

By choosing a suitable loop in the moduli space $\Mbar_{1, 1}$
starting at the point corresponding to $(X, p)$ around the point
corresponding to the nodal elliptic curve, we obtain a deformation of
of $X$ to itself which leaves the even cohomology invariant while it
acts on $H^1(X, \CC)$ via
\begin{equation*}
  \begin{pmatrix}\alpha \\ \beta\end{pmatrix} \mapsto \phi\left(\begin{pmatrix}\alpha \\ \beta\end{pmatrix}\right)
  := \begin{pmatrix}1 & 0 \\ 1 & 1\end{pmatrix} \begin{pmatrix}\alpha \\ \beta\end{pmatrix}.
  = \begin{pmatrix}\alpha \\ \alpha + \beta\end{pmatrix}.
\end{equation*}
In fact, the complete monodromy group acts trivially on the even
cohomology and via the standard $\mathrm{SL}_2(\ZZ)$-representation on
$H^1(X, \CC) \cong \CC^2$.

Because of the deformation invariance of Gromov--Witten theory,
applying this transformation to \emph{all} the descendent insertions
leaves the GW-class invariant.
This gives a relation between GW-classes.

We will use only these relations to establish the vanishing of
unbalanced classes in Section~\ref{sec:unbalanced}.

\medskip

For the proof of Theorem~\ref{thm:red}, we will consider certain
linear combinations of these relations which have a nice form if one
assumes that the vanishing of GW-classes of unbalanced classes has
already been shown.
Let $I$ and $J$ be index sets of the same order and
\begin{equation*}
  \mathbf n\colon I \to \Psi_{\QQ} \qquad
  \mathbf m\colon J \to \Psi_{\QQ}
\end{equation*}
be refined descendent assignments.
Here, a refined descendent assignment is a formal $\QQ$-linear
combination of usual descendent assignments.
Monomials of descendents with such assignments are just expanded
multilinearly.
Refined descendent assignments will only serve as a formal tool here.
We consider the resulting GW-classes to lie in the $\QQ$-vector space
\begin{equation*}
  \bigoplus_{g \ge 0} H^\star(\Mbar_{g, n + \ell(\eta)}).
\end{equation*}
Generalizing the definition of the map $\iota_*$ suitably, we can apply
the degeneration formula also to GW-classes involving refined
descendent assignments.

For a subset $\delta \subset I$ let $S(\delta)$ be the set of all
subsets of $I \sqcup J$ of cardinality $|I|$ containing $\delta$.
For any $D \subseteq I \sqcup J$ we may consider the class
\begin{equation*}
  \tau_{\mathbf n, \mathbf m}(D) := \prod_{i \in I} \tau_{n_i}(\gamma_i^D) \prod_{j \in J} \tau_{m_j}(\gamma_j^D),
\end{equation*}
where
\begin{equation*}
  \gamma_k^D =
  \begin{cases}
    \alpha, & \text{if } k \in D, \\
    \beta, & \text{otherwise.}
  \end{cases}
\end{equation*}
Finally, we will consider a monomial
\begin{equation*}
  N = \prod_{h \in H} \tau_{o_h}(1) \prod_{h' \in H'} \tau_{o'_{h'}}(\omega)
\end{equation*}
in the monodromy invariant insertions.
\begin{proposition}
  The monodromy relation $R(N, \mathbf n, \mathbf m, \delta) = 0$
  holds for any proper subset $\delta \subset I$.
  Here
  \begin{equation*}
    R(N, \mathbf n, \mathbf m, \delta) = \sum_{D \in S(\delta)} \left[ N \tau_{\mathbf n, \mathbf m}(D)\right]_d^X.
  \end{equation*}
\end{proposition}
\begin{proof}
  Consider the application of the monodromy transform $\phi$ to
  \begin{equation*}
    \left[ N \prod_{i \in I} \tau_{n_i}(\gamma_i^\delta) \prod_{j \in J} \tau_{m_j}(\beta)\right]_d^X.
  \end{equation*}
  This class vanishes since it is unbalanced because
  $\delta \subset I$ is a proper subset.
  After applying $\phi$, all terms but those with exactly $|I|$
  insertions of $\alpha$ vanish.
  The sum of these remaining terms is exactly
  $R(N, \mathbf n, \mathbf m, \delta)$.
\end{proof}

\subsection{from the elliptic action}
\label{sec:ell}

Using the group structure of $X$ induced by identifying $X$ with its
Jacobian via a point $0 \in X$ gives another set of relations.

Let the \emph{small diagonal} of $X^r$ be the subset
\begin{equation*}
  \{(x, \dotsc, x): x\in X\} \subset X^r
\end{equation*}
and $\Delta_r \in H^{r - 1}(X^r, \CC)$ be its Poincaré dual.
We will use the fact that $\Delta_r$ is invariant under the diagonal
action of the elliptic curve $X$ on $X^r$, and the Künneth
decomposition of $\Delta_r$ to obtain the relations.

Let $K$ and $H$ be two ordered index sets and $P$ a set partition of
$K$ into subsets of size at least 2.
For any part $p$ of $P$, we have a product evaluation map
\begin{equation*}
  \phi_p\colon \Mbar_{g, K \sqcup H}(X, d) \to X^{|p|}.
\end{equation*}
Let $\mathbf l\colon K \to \Psi$ be an assignment of descendents.
Finally, let $M$ be a monomial in insertions of the identity
\begin{equation*}
  M = \prod_{h \in H} \tau_{o_h}(1).
\end{equation*}
\begin{proposition}
  The elliptic vanishing relation $V(M, P,\mathbf l) = 0$ holds.
  Here
  \begin{equation}
    \label{eq:V}
    V(M, P, \mathbf l) := \pi_*\left(\prod_{h \in H} \psi_h^{o_h} \prod_{k \in K} \psi_k^{l_k} \prod_{p \in P}\phi_p^*(\Delta_{|p|}) \cap [\Mbar_{g, K \sqcup H}(X, d)]^\vir\right).
  \end{equation}
\end{proposition}
Notice that no insertions of $\omega$ appear and that we do not work
in the relative theory.
There is a natural generalization to a more general assignment
$\mathbf l\colon K \to \Psi_{\QQ}$.
\begin{proof}
  The elliptic curve $X$ acts on the moduli space
  $\Mbar_{g, H \sqcup K}(X, d)$ by the action induced from the group
  operation $X \times X \to X$.
  The action can be used to fix the image in $X$ of one marked point
  $q$.
  This gives an $X$-equivariant splitting
  \begin{equation*}
    \Mbar_{g, H \sqcup K}(X, d) \cong \ev_q^{-1}(0) \times X.
  \end{equation*}
  In particular, there exists an algebraic quotient
  \begin{equation*}
    \Mbar_{g, H \sqcup K}(X, d) / X \cong \ev_q^{-1}(0)
  \end{equation*}
  of $\Mbar_{g, H \sqcup K}(X, d)$.

  Notice that the products in \eqref{eq:V} are pulled back via the
  projection map
  \begin{equation*}
    \Mbar_{g, H \sqcup K}(X, d) \to \Mbar_{g, H \sqcup K}(X, d) / X
  \end{equation*}
  from an analogous class on the quotient space.
  Furthermore, the virtual fundamental class is also pulled back from
  the quotient.
  Thus, the push-pull formula applied to the projection map implies
  that the GW-class must vanish.
\end{proof}

To apply these relations, we need to reformulate them as relations
between GW-classes of $X$.
In order to rewrite the $\phi_p$-pullbacks as products of usual
pullbacks via the evaluation maps, we Künneth-decompose the classes
$\Delta_r$.
For $\Delta_2$ and $\Delta_3$ we have for example
\begin{align*}
  \Delta_2 =& 1 \otimes \omega + \omega \otimes 1 - \alpha \otimes \beta + \beta \otimes \alpha, \\
  \Delta_3 =& 1 \otimes \omega \otimes \omega + \omega \otimes 1 \otimes \omega + \omega \otimes \omega \otimes 1 - \omega \otimes \alpha \otimes \beta + \omega \otimes \beta \otimes \alpha\\
  &-\alpha \otimes \omega \otimes \beta + \beta \otimes \omega \otimes \beta - \alpha \otimes \beta \otimes \omega + \beta \otimes \alpha \otimes \omega.
\end{align*}
In general, $\Delta_r$ is a sum
$\Delta_r = \Delta_r^{even} + \Delta_r^{odd}$ where $\Delta_r^{even}$
is the sum of the $r$ classes of the form
\begin{equation*}
  \omega \otimes \dots \otimes \omega \otimes 1 \otimes \omega \dots \otimes \omega
\end{equation*}
and $\Delta_r^{odd}$ is the sum of the $\binom{r}{2}$ linear
combinations of classes
\begin{align*}
  -\omega \otimes \dots \otimes \omega \otimes \alpha \otimes \omega \otimes \dots \otimes \omega \otimes \beta \otimes \omega \otimes \dots \otimes \omega\\
  +\omega \otimes \dots \otimes \omega \otimes \beta \otimes \omega
  \otimes \dots \otimes \omega \otimes \alpha \otimes \omega \otimes
  \dots \otimes \omega.
\end{align*}
We will mostly be interested in the odd summand since the even summand
will usually already be known by an induction hypothesis.

\section{Unbalanced classes}
\label{sec:unbalanced}

In this section, let us fix a monomial $M$ in insertions of even
classes of $X$ and an index set $I$ for the odd insertions.
Our aim is to show that the classes
\begin{equation*}
  C(S) := \left[ M \cdot \prod_{i \in I} \tau_{n_i}(\gamma^{I \setminus S}_i) \Big| \eta\right]^X_r,
\end{equation*}
where
\begin{equation*}
  \gamma^{I \setminus S}_i :=
  \begin{cases}
    \alpha & \text{if }i \notin S \\
    \beta & \text{if }i \in S
  \end{cases}
\end{equation*}
vanish in the unbalanced case, that is when $|I| \neq 2|S|$.

By symmetry, for the proof of the vanishing $C(S) = 0$, we can assume
that $2|S| < |I|$.
We then proceed by induction on $|S|$, starting with the empty case
$|S| < 0$.

Choose a subset $J \subset I$ of size $2|S| + 1$ containing $S$.
For any subset $T \subset J$ of size $|S| + 1$, consider the monodromy
relation
\begin{equation*}
  C(T)
  = \left[ M \cdot \prod_{i \in I} \tau_{n_i}(\phi(\gamma^{I \setminus T}_i)) \Big| \eta \right]^X_r
  = \sum_{S' \subset T} C(S').
\end{equation*}
After subtracting the left-hand side from the right-hand side and
using the induction hypothesis, we obtain the relation
\begin{equation}
  \label{eq:unbrel}
  0 = R(T) := \sum_{\substack{S' \subset T \\ |S'| = |S|}} C(S').
\end{equation}
The relation \eqref{eq:unbrel} can be inverted as
\begin{equation*}
  C(S') = \sum_{i = 0}^{|S|} (-1)^{i + |S|} c_i^{-1} \sum_{\substack{T \subset J \\ |T| = |S| + 1, |T \cap S'| = i}} R(T)
\end{equation*}
where
\begin{equation*}
  c_i := (|S| + 1) \binom{|S|}i \neq 0.
\end{equation*}
Therefore, $C(S') = 0$ for any $S' \subset J$ of size $|S|$.
In particular, we have established the induction step $C(S) = 0$.

\section{Balanced classes}
\label{sec:balanced}

In this section, we will finish the proof of Theorem~\ref{thm:red} in
the remaining case of balanced classes, therefore giving a proof of
Theorem~\ref{thm:thm}.
We follow the discussion of \cite[Section~5.5]{OkPa06c} and try to
keep the notation as similar as possible.
Compared to \cite{OkPa06c}, there is one additional induction on the
codimension.

The following lemma will be used to determine relative GW-classes from
a set of related absolute GW-classes.
Before stating the lemma, we need to introduce a special refined
descendent assignment.

Let $P(d)$ be the set of partitions of $d$ and $\QQ^{P(d)}$ the
$\QQ$-vector space of functions from $P(d)$ to $\QQ$.
Let
\begin{equation*}
  \tilde \tau(\omega) = \sum_{q = 0}^\infty c_q \tau_q(\omega)
\end{equation*}
be a refined descendent of $\omega$.
The Gromov--Witten theory of $\PP^1$ relative to a point gives for
each $v \ge 0$ a function
\begin{equation*}
  \gamma_v\colon P(d) \to \QQ, \quad \eta \mapsto \left\langle \tilde\tau(\omega)^v | \eta\right\rangle^{\PP^1}.
\end{equation*}

\begin{fact}
  There exists a $\QQ$-linear combination $\tilde \tau(\omega)$
  depending on $d$ such that the set of functions
  \begin{equation*}
    \{\gamma_0, \gamma_1, \dotsc\}
  \end{equation*}
  spans $\QQ^{P(d)}$.
\end{fact}
\begin{proof}
  This is Lemma~5.6 in \cite{OkPa06c}.
  Its proof uses the Gromov--Witten Hurwitz correspondence
  \cite{OkPa06a}.
\end{proof}

We will fix such a refined descendent assignment
$\tilde \tau(\omega)$.
Let us define
\begin{equation*}
  \tilde \psi = \sum_{q = 0}^\infty c_q \psi^q
\end{equation*}
so that formally $\tilde \tau(\omega) = \tau_{\tilde \psi}(\omega)$.

\begin{lemma}
  \label{lem:absrel}
  Let $M$, $L$, $A$, $B$ be monomials in insertions of $1$, $\omega$,
  $\alpha$ and $\beta$ respectively,
  \begin{equation*}
    M = \prod_{h \in H} \tau_{o_h}(1), \quad
    L = \prod_{h' \in H'} \tau_{o'_{h'}}(\omega), \quad
    A = \prod_{i \in I} \tau_{n_i}(\alpha), \quad
    B = \prod_{j \in J} \tau_{m_j}(\beta),
  \end{equation*}
  and let $\eta \in P(d)$ be a splitting.
  Then the GW-classes
  \begin{equation*}
    [MAB | \eta]_r^X, \quad
    [MLAB]_{r, d}^X
  \end{equation*}
  are tautological if the classes
  \begin{equation*}
    [M'\tilde\tau(\omega)^v AB]_{r, d}^X, \quad
    [M' AB | \mu]_{r'}^X
  \end{equation*}
  are tautological for arbitrary $v \ge 0$, $r' \le r$, $\mu \in P(d)$
  and divisors $M'$ of $M$ except (possibly) in the case $r'
  = r$, $M' = M$.
\end{lemma}
\begin{proof}
  We first study the case $M = 1$, $r = 0$.
  There is a degeneration of $X$ into $X \cup_{\mathbf{pt}} \PP^1$ we
  have already studied in Section~\ref{sec:red}.
  The corresponding degeneration formula spells here
  \begin{equation*}
    [\tilde\tau(\omega)^v AB]_{0, d}^X =
    \sum_{|\eta| = d} \mathfrak z(\eta) \iota_*\left([AB | \eta]_0^X , \left[ \tilde\tau(\omega)^v | \eta\right]_0^{\PP^1}\right).
  \end{equation*}
  By the Fact, letting $v$ vary this determines $[AB | \eta]_0^X$ for
  all $\eta$.
  The degeneration formula
   \begin{equation*}
     [LAB]_{0, d}^X = \sum_{|\eta| = d} \mathfrak z(\eta)
     \iota_\star\left( [LAB | \eta]_0^X, \left[
         \tilde\tau(\omega)^v | \eta\right]_0^{\PP^1}\right).
  \end{equation*}
  then determines the second kind of GW-class if $M = 1$, $r = 0$.

  In general, there are additional sums in the degeneration formula:
  one for the distribution of the factors of $M$ and one for the
  splitting of the domain curve.
  However by the hypothesis of the lemma and the fact that we already
  have shown the tautologicalness of GW-classes of $\PP^1$, only the
  summand corresponding to the distribution of all of $M$ to $X$ and
  all of $r$ to $X$ may be non-tautological.
  But then we can mirror the above argument in the simple case.
\end{proof}

\subsection{Simple case}

To illustrate the principle of the proof, we start with the GW-classes
with only two odd insertions (one of each $\alpha$ and $\beta$).
So for descendent assignments $n$, $m$, a monomial of identity
insertions
\begin{equation*}
  M = \prod_{h \in H} \tau_{o_h}(1)
\end{equation*}
and the choice of splitting $\mu$ for the relative point, we wish to
determine
\begin{equation*}
  \left[ M \tau_n(\alpha) \tau_m(\beta) | \mu\right]_r^X
\end{equation*}
in terms of GW-classes with only even insertions.
By induction on $r$ and $M$, we will assume that this statement has
already been proven for all $r' \le r$ and $M' | M$ except (possibly)
in the case $r' = r$, $M' = M$.

Let $K_v$ be an index set with $v + 2$ elements.
We first look at the elliptic vanishing relation
$V(M, \{K_v\}, \mathbf l)$ where $\mathbf l$ assigns $\tilde \psi$ to
every element of $K_v$.
The relation contains $2\binom{v + 2}{2}$ summands which contain odd
classes, and in fact since the descendent assignment is identical for
each element of $K_v$, each of them is equal to
\begin{equation*}
  -\left[ M \tilde\tau(\omega)^v \tilde\tau(\alpha) \tilde\tau(\beta)\right]_{r, d}^X,
\end{equation*}
which we thus have determined in terms of even GW-classes.

Lemma~\ref{lem:absrel} and the induction hypothesis yield the
determination of the classes
\begin{equation}
  \label{eq:easy}
  \left[ M \tilde\tau(\alpha) \tilde\tau(\beta) | \eta\right]_r^X,\ \left[ M L \tilde\tau(\alpha) \tilde\tau(\beta)\right]_{r, d}^X
\end{equation}
for any monomial $L$ in descendents of $\omega$.

Next, we look at the elliptic vanishing relation
$V(M, \{K_v\}, \mathbf l)$ where this time the descendent assignment
$\mathbf l$ takes the value $\tilde \psi$ at all but the first element
of $K_v$ where it takes the value $\psi^n$.
The even terms are still of no relevance but now there are four kinds
of odd summands.
They are
\begin{align*}
  -(v + 1) \left[ M \tilde\tau(\omega)^v \tau_n(\alpha) \tilde\tau(\beta)\right]_{r, d}^X  \\
  + (v + 1) \left[ M \tilde\tau(\omega)^v \tau_n(\beta) \tilde\tau(\alpha)\right]_{r, d}^X \\
  -\binom{v + 1}{2}\left[ M \tilde\tau(\omega)^{v - 1}\tau_n(\omega)
    \tilde\tau(\alpha) \tilde\tau(\beta)\right]_{r, d}^X \\
  + \binom{v + 1}{2}\left[ M \tilde\tau(\omega)^{v - 1} \tau_n(\omega)
    \tilde\tau(\beta) \tilde\tau(\alpha)\right]_{r, d}^X.
\end{align*}
We are only interested in the first pair of summands since the second
two are determined by \eqref{eq:easy}.
By applying the relation
$R(M \tilde\tau(\omega)^v, \{\psi^n\}, \{\tilde\psi\}, \emptyset)$, we
see that the first two summands are equal.
Therefore, we now know
\begin{equation*}
  \left[ M \tilde\tau(\omega)^v \tau_n(\alpha) \tilde\tau(\beta)\right]_{r, d}^X
\end{equation*}
and by Lemma~\ref{lem:absrel} also
\begin{equation}
  \label{eq:left}
  \left[ M \tau_n(\alpha) \tilde\tau(\beta) | \eta\right]_r^X, \left[ M L \tau_n(\alpha) \tilde\tau(\beta)\right]_{r, d}^X.
\end{equation}

Repeating this argument, we successively determine
\begin{align}
  \left[ M \tilde\tau(\alpha) \tau_m(\beta) | \eta\right]_r^X, \left[ M L \tilde\tau(\alpha) \tau_m(\beta)\right]_{r, d}^X, \label{eq:right}\\
  \left[ M \tau_n(\alpha) \tau_m(\beta) | \eta\right]_r^X \label{eq:full}.
\end{align}
For \eqref{eq:right}, we need the elliptic vanishing relation
$V(M, \{K_v\}, \mathbf l)$, where $\mathbf l$ takes the value
$\tilde\psi$ on all but the last elements of $K_v$ where it is
$\psi^m$.
As before, two terms in this relation are not yet determined and these
are proportional to each other by the monodromy relation
$R(M \tilde\tau(\omega)^v, \{\tilde\psi\}, \{\psi^m\}, \emptyset)$.

For \eqref{eq:full}, we use the relation $V(M, \{K_v\}, \mathbf l)$
with $\mathbf l$ having the value $\tilde\psi$ on all but the first
and the last element of $K_v$ where it takes the values $n$ and $m$,
respectively.
To see that there is only a pair of not yet determined terms, we in
particular need to use \eqref{eq:left} and \eqref{eq:right}.
We finish with the use of the relation
$R(M \tilde\tau(\omega)^v, \{\psi^n\}, \{\psi^m\}, \emptyset)$.

\subsection{General case}

Let $I$ and $J$ be two ordered index sets of the same size and let
\begin{equation*}
  \mathbf n\colon I \to \Psi_\QQ, \quad \mathbf m\colon J \to \Psi_\QQ
\end{equation*}
be general descendent assignments.
In order to prove Theorem~\ref{thm:red}, we need to calculate for a
monomial $M$ in insertions of the identity, the GW-classes
\begin{equation*}
  \left[ M \prod_{i \in I} \tau_{n_i}(\alpha) \prod_{j \in J} \tau_{m_j}(\beta) | \eta\right]_{r, d}^X
\end{equation*}
in terms of lower GW-classes.
This follows from the following lemma.
\begin{lemma}
  For $s, t \ge 0$ the GW-classes
  \begin{align*}
    \left[ M \prod_{i \le s} \tau_{n_i}(\alpha) \prod_{s < i \in I} \tilde\tau(\alpha) \prod_{J \ni j \le t} \tau_{m_j}(\beta) \prod_{t < j} \tilde\tau(\beta) | \eta\right]_r^X, \\
    \left[ ML \prod_{i \le s} \tau_{n_i}(\alpha) \prod_{s < i \in I} \tilde\tau(\alpha) \prod_{J \ni j \le t} \tau_{m_j}(\beta) \prod_{t < j} \tilde\tau(\beta)\right]_{r, d}^X,
  \end{align*}
  for an arbitrary monomial $L$ in insertions of point classes
  $\omega$ are determined in terms of the GW-classes with strictly
  less insertions as well as
  \begin{align}
    \left[ M' \prod_{i \le s'} \tau_{n_i}(\alpha) \prod_{s' < i \in I} \tilde\tau(\alpha) \prod_{J \ni j \le t'} \tau_{m_j}(\beta) \prod_{t' < j} \tilde\tau(\beta) | \eta\right]_{r'}^X, \nonumber \\
    \left[ M'L' \prod_{i \le s'} \tau_{n_i}(\alpha) \prod_{s' < i \in I} \tilde\tau(\alpha) \prod_{J \ni j \le t'} \tau_{m_j}(\beta) \prod_{t' < j} \tilde\tau(\beta)\right]_{r', d}^X, \label{eq:genpt}
  \end{align}
  where $L'$ is an arbitrary monomial in insertions of $\omega$ and we
  have $(r', s', t', M') < (r, s, t, M)$.
  Here we have used the partial order defined by
  $(r', s', t', M') \le (r, s, t, M)$ if and only if $r' \le r$,
  $s' \le s$, $t' \le t$ and $M' | M$.
\end{lemma}
\begin{proof}
  We need additional notation.
  For $v \ge 0$ let $W$ be an index set of cardinality $v$.
  Define $K_v$ by
  \begin{equation*}
    K_v = I \sqcup W \sqcup J
  \end{equation*}
  with order implicit in the notation.
  Let $\mathbf l_{f[s]l[t]}\colon K_v \to \Psi_\QQ$ be the descendent
  assignment with
  \begin{equation*}
    \mathbf l_{f[s]l[t]}(k) =
    \begin{cases}
      n_k, & \text{if }k\text{ is one of the first }s\text{ elements of }I, \\
      m_k, & \text{if }k\text{ is one of the first }t\text{ elements of }J, \\
      \tilde\psi, & \text{else.}
    \end{cases}
  \end{equation*}
  We call the $s$ first elements of $I \subset K_v$ and the $t$ first
  elements of $J \subset K_v$ \emph{special} elements of $K_v$ with
  respect to $(s, t)$.

  Let $\sigma\colon I \to J$ be a bijection, which we can, using the
  orders on $I$ and $J$, also interpret as a permutation of $I$.
  Let $P_\sigma$ be the set partition of $K_v$ with first part
  $\{1, \sigma(1)\} \cup W$ and pairs $\{i, \sigma(i)\}$ as the other
  parts.

  Consider the relations $V(M, P_\sigma, \mathbf l_{f[s], l[t]})$ for
  varying $\sigma$.
  By the induction hypothesis, we only need to consider the terms from
  the Künneth decomposition with exactly $|I| + |J|$ odd insertions.
  After expanding the product, there are
  $2 \cdot \binom{v + 2}{2} \cdot 2^{|I| - 1}$ terms of this kind.
  If we consider the odd part of the Künneth decomposition
  corresponding to the part $\{1, \sigma(1)\} \cup W$ of $P$ in more
  detail, we see that depending on the $s$, $t$ and $\sigma(1)$ still
  different kinds of terms might occur.
  We only need to take into account the terms with the least possible
  amount of point classes $\omega$ distributed to the special elements
  of $K_v$ with respect to $(s, t)$ since all possible other terms
  are of the form \eqref{eq:genpt} for
  \begin{equation*}
    (s', t') \in \{(s - 1, t), (s, t - 1), (s - 1, t - 1)\}.
  \end{equation*}
  The remaining terms occur still with a combinatorial multiplicity
  $C_\sigma$ depending on the number of special elements in
  $\{1, \sigma(1)\}$.
  These multiplicities are
  \begin{equation*}
    C_\sigma =
    \begin{cases}
      1, & \text{if }\{1, \sigma(1)\}\text{ contains 2 special elements,} \\
      v + 1, & \text{if }\{1, \sigma(1)\}\text{ contains 1 special elements,} \\
      \binom{v + 2}{2}, & \text{if }\{1, \sigma(1)\}\text{ contains 0
        special elements.}
    \end{cases}
  \end{equation*}
  The last case can only occur if $s = 0$.

  Let $V$ be the relation obtained by summing these relations over
  all permutations $\sigma$ and weighting with $C_\sigma^{-1}$ and a
  sign
  \begin{equation*}
    \sum_{\sigma} (-1)^{\binom{|I|}2} \sign(\sigma) C_\sigma^{-1} V(M, P_\sigma, \mathbf l_{f[s], l[t]}),
  \end{equation*}
  and removing terms determined by the induction hypothesis or of the
  form \eqref{eq:genpt} for $(s', t')$ as before.
  Using the notation from Section~\ref{sec:mono}, we can write
  \begin{multline*}
    V = \sum_{\delta \subset I} \sum_{D \in S^*(\delta)} (-1)^{|I| - |\delta|}|\delta|!(|I| - |\delta|)! \\
    \Big[ M \tilde\tau(\omega)^v \prod_{i \le s} \tau_{n_i}(\gamma_i^D) \prod_{s < i \in I} \tilde\tau(\gamma_i^D) \prod_{J \ni j \le t} \tau_{m_j}(\gamma_j^D) \prod_{t < j} \tilde\tau(\gamma_j^D) \Big]_{r, d}^X,
  \end{multline*}
  where $S^*(\delta)$ denotes the set of all subsets of $I \sqcup J$
  such that $D \cap I = \delta$.
  Using the substitution
  \begin{multline*}
    e_k = \sum_{|\delta| = k} \sum_{D \in S^*(\delta)} \\
    \Big[ M \tilde\tau(\omega)^v \prod_{i \le s} \tau_{n_i}(\gamma_i^D) \prod_{s < i \in I} \tilde\tau(\gamma_i^D) \prod_{J \ni j \le t} \tau_{m_j}(\gamma_j^D) \prod_{t < j} \tilde\tau(\gamma_j^D) \Big]_{r, d}^X,
  \end{multline*}
  we can write $V$ more compactly as
  \begin{equation*}
    V = \sum_{k = 0}^{|I|} (-1)^{|I| - k}k!(|I| - k)!\ e_k.
  \end{equation*}
  We wish to eliminate $e_0, \dotsc, e_{|I| - 1}$ from $V$ in order to
  obtain a formula for
  \begin{equation*}
    e_{|I|} = \left[ M \tilde\tau(\omega)^v \prod_{i \le s} \tau_{n_i}(\alpha) \prod_{s < i \in I} \tilde\tau(\alpha) \prod_{J \ni j \le t} \tau_{m_j}(\beta) \prod_{t < j} \tilde\tau(\beta)\right]_{r, d}^X.
  \end{equation*}

  Let $R(\ell)$ be the sum
  \begin{equation*}
    R(\ell) = \sum_{|\delta| = \ell} R(M \tilde\tau(\omega)^v, \mathbf n', \mathbf m', \delta),
  \end{equation*}
  where here $\mathbf n'$ and $\mathbf m'$ are the restrictions of
  $\mathbf l_{f[s], l[t]}$ to $I$ and $J$, respectively.
  Since unbalanced GW-classes vanish, we have the expansion
  \begin{multline*}
    R(\ell) = \sum_{|\delta| \ge \ell} \sum_{D \in S^*(\delta)} \binom{|\delta|}\ell \\
    \Big[ M \tilde\tau(\omega)^v \prod_{i \le s}
    \tau_{n_i}(\gamma_i^D) \prod_{s < i \in I} \tilde\tau(\gamma_i^D)
    \prod_{J \ni j \le t} \tau_{m_j}(\gamma_j^D) \prod_{t < j}
    \tilde\tau(\gamma_j^D)
    \Big]_{r, d}^X \\
    = \sum_{k \ge \ell} \binom k\ell e_k.
  \end{multline*}
  The following lemma in linear algebra gives us the formula for the
  desired $e_{|I|}$

  \begin{lemma}
    \label{lem:linalg}
    Let $e_0, \dotsc, e_n$ be a basis of the vector space
    $\QQ^{n + 1}$.
    Then the vectors
    \begin{equation*}
      V = \sum_{k = 0}^n (-1)^{n - k}k!(n - k)! e_k
    \end{equation*}
    and
    \begin{equation*}
      R(\ell) = \sum_{k \ge \ell} \binom{k}{\ell} e_k
    \end{equation*}
    for $0 \le \ell < n$ form a basis of $\QQ^{n + 1}$.
  \end{lemma}
  \begin{proof}
    Note that by formally extending the definition of $R(\ell)$ to
    $R(n)$, we obtain an $(n + 1) \times (n + 1)$ lower uni-triangular
    matrix $R$ with coefficients
    \begin{equation*}
      R_{ab} = \binom{a}{b}.
    \end{equation*}
    $R$ is therefore invertible and the coefficients of its inverse
    $R^{-1}$ are
    \begin{equation*}
      (R^{-1})_{ab} = (-1)^{a + b} \binom{a}{b}.
    \end{equation*}
    In particular, the $R(0), \dotsc, R(n - 1)$ are linearly
    independent.
    In order to show that $V$ is not a linear combination of these
    vectors, we expand $V$ in terms of the basis corresponding to $R$
    \begin{equation*}
      V = \sum_{\ell = 0}^n c_\ell R(\ell)
    \end{equation*}
    and check that the coefficient $c_n$ is nonzero:
    \begin{equation*}
      c_n = \sum_{k = 0}^n (-1)^{n + k} \binom{n}{k} (-1)^{n - k}k!(n - k)! = (n + 1)!
    \end{equation*}
  \end{proof}

  We next apply Lemma~\ref{lem:absrel} to determine
  \begin{equation*}
    \left[ M \prod_{i \le s} \tau_{n_i}(\alpha) \prod_{s < i \in I} \tilde\tau(\alpha) \prod_{J \ni j \le t} \tau_{m_j}(\beta) \prod_{t < j} \tilde\tau(\beta) | \eta\right]_r^X
  \end{equation*}
  using the induction hypothesis for the $r$ induction.

  By a degeneration argument as in the simple case, we finally obtain a
  formula for
  \begin{equation*}
    \left[ ML \prod_{i \le s} \tau_{n_i}(\alpha) \prod_{s < i \in I} \tilde\tau(\alpha) \prod_{J \ni j \le t} \tau_{m_j}(\beta) \prod_{t < j} \tilde\tau(\beta)\right]_{r, d}^X.
  \end{equation*}
\end{proof}

\printbibliography
\addcontentsline{toc}{section}{References}

\vspace{+8 pt}
\noindent
Department of Mathematics, University of Michigan, 2074 East Hall, 530 Church Street, Ann Arbor, MI 48109, USA \\
janda@umich.edu

\end{document}